\documentclass[12pt]{amsart}
\usepackage{amscd,amssymb,graphics}

\usepackage{amsfonts}
\usepackage{amsmath}
\usepackage{amsxtra}
\usepackage{latexsym}
\usepackage[mathcal]{eucal}

\usepackage{graphics,colortbl}

\usepackage[final]{pdfpages}

\usepackage[pdftex,bookmarks,colorlinks,breaklinks]{hyperref}

\input xy
\xyoption{all}
\usepackage{epsfig}

\swapnumbers
\theoremstyle{plain}
\newtheorem{thm}{Theorem}[section]
\newtheorem{cor}[thm]{Corollary}
\newtheorem{lem}[thm]{Lemma}
\newtheorem{prop}[thm]{Proposition}

\theoremstyle{definition}
\newtheorem{defn}[thm]{Definition}

\newtheorem{rmk}[thm]{Remark}

\newtheorem{exa}[thm]{Example}
\newtheorem{exas}[thm]{Examples}
\newtheorem{prob}[thm]{Problem}

\numberwithin{thm}{section}
\numberwithin{equation}{section}

\newcommand{\delete}[1]{} 

\newcommand{\ben}{\begin{enumerate}}
	
	\newcommand{\een}{\end{enumerate}}
\newcommand{\bit}{\begin{itemize}}
	
	\newcommand{\eit}{\end{itemize}}

\def\ep{{\varepsilon}}
\def\al{\alpha}
\def\om{\omega}
\def\Om{\Omega}

\newcommand{\Ga}{\Gamma}
\newcommand{\del}{\delta}

\newcommand{\la}{\lambda}
\newcommand{\La}{\Lambda}

\newcommand{\card}{\rm{card\,}}

\def\R {{\mathbb R}}
\def\N {{\mathbb N}}
\def\Z {{\mathbb Z}}

\newcommand{\br}{\vspace{4 mm}}

\newcommand{\cls}{{\rm{cls\,}}}

\def\Id{\rm{Id}}

\def\QED{\nobreak\quad\ifmmode\roman{Q.E.D.}\else{\rm Q.E.D.}\fi}

\newenvironment{psmallmatrix}
  {\left(\begin{smallmatrix}}
  {\end{smallmatrix}\right)}

\begin{document}

\title[On weak rigidity and weakly mixing enveloping semigroups]{On weak rigidity and weakly mixing enveloping semigroups}

\author{Ethan Akin, Eli Glasner and Benjamin Weiss}

\address{The City College\\
     137 Street and Convent Avenue,
 New York City\\
         NY 10031\\
         USA}
\email{ethanakin@earthlink.net}

%
%
%

\address{Department of Mathematics\\
     Tel Aviv University\\
         Tel Aviv\\
         Israel}
\email{glasner@math.tau.ac.il}

\address {Institute of Mathematics\\
 Hebrew University of Jerusalem\\
Jerusalem\\
 Israel}
\email{weiss@math.huji.ac.il}

\dedicatory{In memory of Sergii Kolyada, colleague and friend}


\begin{date}
{November 05, 2017}
\end{date}


\begin{abstract}
The question we deal with here, which was presented to us by Joe Auslander and Anima Nagar,
is whether there is a nontrivial cascade $(X,T)$ whose enveloping semigroup, as a dynamical system, is topologically weakly mixing (WM).
After an introductory section recalling some definitions and classic results,
we establish some necessary conditions for this to happen, and in the final section
we show, using Ratner's theory, that the
enveloping semigroup of the `time one map' of a classical horocycle flow is weakly mixing.	
\end{abstract}



\maketitle

%


\section{Introduction}

A \emph{cascade} is a homeomorphism $T$ on a compact Hausdorff space $X$.
We call the system $(X,T)$ metric when $X$ is metrizable. If $A$ is a closed,
invariant subset of $X$,
i.e. $T(A) = A$,
then the restriction of $T$ to $A$ defines the
\emph{subsystem} on $A$. If  $(X_1,T_1)$ is a cascade and $\pi : X \to X_1$ is a
continuous surjection such that $\pi \circ T = T_1 \circ \pi$ then
$\pi$ is a surjective \emph{cascade morphism} and $(X_1,T_1)$ is a \emph{factor} of $(X,T)$.

For $A,B \subset X$
we let $N(A,B) = \{ i \in \Z : A \cap T^{-i}(B) \not= \emptyset \}$.
Here $\Z$ denotes the set of integers and we use $\N$ for the set of non-negative integers.
We write $N(x,B)$ for $N(\{ x \},B)$ when $x \in X$. The cascade is
\emph{transitive} when $N(U,V)$ is nonempty for every pair of nonempty, open subsets $U,V$ of $X$.
We let $\mathcal{O}(x)=\{T^nx : n \in \Z\}$ denote the {\em orbit} of $x$.
A point $x$ is a \emph{transitive point}
when the orbit-closure  $\overline{\mathcal{O}(x)} = X$. When $(X,T)$
admits transitive points the system is called \emph{point transitive}.
A point transitive system is transitive and the converse holds when the system is metric.
In a metric transitive system the set of transitive points
forms a dense $G_\del$ set by the Baire Category Theorem .
The system $(X,T)$ is called {\em totally transitive} if for every $0 \not= n \in \Z$
the system $(X,T^n)$ is transitive.

When the homeomorphism is understood, we will refer to the system $X$.

The system $X$ is called \emph{minimal} when every point of $X$ is a transitive point, or, equivalently, when
$X$ contains no proper, closed, invariant subset. A point $x \in X$ is called a \emph{minimal point}, or an \emph{almost periodic point}
(hereafter abbreviated a.p.) when the restriction of $T$ to  $\overline{\mathcal{O}(x)}$ is minimal.

A subset of $\Z$ is called {\em thick} when it contains blocks of consecutive integers of arbitrary length. It is called \emph{syndetic} when it meets every thick subset.
A point $x$ is minimal if and only if $N(x,U)$ is syndetic for every neighborhood $U$ of $x$, see, e.g. \cite{Aus} Theorem 1.8.

 An \emph{ambit}
$(X,x_0,T)$ is a point transitive cascade $(X,T)$
with a chosen transitive base point $x_0$. If $\pi : (X,T) \to (X_1,T_1)$
is a surjective cascade morphism and $(X,x_0,T)$
is an ambit then $(X_1,\pi(x_0),T_1)$ is an ambit factor.

The enveloping semigroup $E(X,T)$
(or simply $E(X)$) is the closure in $X^X$ of the set $\{ T^i : i \in \Z \}$.
If $T_*$ denotes composition with $T$
then $(E(X),\Id_X,T_*)$ is an ambit, where $\Id_X = T^0$ is the identity map.
If $x \in X$ then the evaluation map
$p \mapsto px$,
is a surjection of ambits
$ev_x : (E(X),{\Id}_X,T_*) \to (\overline{\mathcal{O}(x)},x,T)$
taking the
enveloping semigroup onto the orbit closure of $x$.

\begin{prop}
For any non-empty index set $I$ and the product system $(X^I,T^{(I)})$
we can identify $E(X)$ with $E(X^I)$ by mapping $T^i$ to $(T^{(I)})^i$.
Hence, for any $k$-tuple $(x_1, x_2, \dots, x_k) \in X^k$ the ambit
$\overline{\mathcal{O}(x_1, x_2, \dots, x_k)}$ is a factor of $E(X)$
via evaluation at $(x_1, x_2, \dots, x_k)$,
i.e. by $p \mapsto (px_1, px_2, \dots, px_k)$ for $p \in E(X)$.
 In addition, $E(X)$ can be expressed
as the inverse limit of these factors.
\end{prop}

\begin{proof}
The first two claims are easy to check. To prove the last one,
observe that
$(E(X),T_*)$ is the orbit closure $(\overline{\mathcal{O}(Id_X)},T^X)$ in $(X^X,T^X)$.
Partially ordered by inclusion, the collection of finite (unordered) subsets of $X$
forms a directed set with respect to which
$\overline{\mathcal{O}(Id_X)}$ is the inverse limit $ \lim_{\leftarrow} \overline{\mathcal{O}(x_1, x_2, \dots, x_k)}$
via the evaluation maps.
\end{proof}

If $A$ is
closed invariant subset of $X$ then the restriction map $E(X) \to E(A)$ is a surjective morphism.
A surjective morphism $\pi : (X,T) \to (X_1,T_1)$
induces a surjective morphism
$\pi_* : E(X,T) \to E(X_1,T_1)$ by mapping $T^i$ to $T_1^i$.

\section{Recurrence}\label{sec-1.5}

For the system $(X,T)$ and the point $x \in X$
the limit point sets $\al(x)$ and $\om(x)$
are the sets of limit points of the bi-infinite sequence $\{ T^i(x) : i \in \Z \}$ as $i$ tends to
$- \infty$ and $+ \infty$ respectively. Both $\al(x)$ and $\om(x)$ are non-empty, closed, invariant subsets.
The orbit-closure of $x$ consists of the orbit and $\al(x) \cup \om(x)$.

For $(E(X),T_*)$ 
we denote
$$
Ad_+(X) =  \bigcap \cls \{T^n : n \ge 1\}, \quad Ad_-(X)= \bigcap \cls \{T^n : n \le -1\}
$$
and $Ad(X) = Ad_+(X) \cup Ad_-(X)$.
$Ad(X)$ is called the \emph{adherence semigroup} of $(X,T)$.
It is a closed
subsemigroup of $E(X)$. Furthermore,
$ev_x(Ad_-(X)) = \al(x)$ and $ev_x(Ad_+(X)) = \om(x)$.

A point $x$ is called
{\em recurrent} if $x \in \al(x) \cup \om(x)$ and {\em positive recurrent} if $x \in  \om(x)$.
Thus, $x$ is recurrent (or positive recurrent) if $x = px$ for some $p \in Ad(X)$ (resp. for some $p \in Ad_+(X)$).

A point is
recurrent if and only if $N(x,U)$ is infinite for
every neighborhood $U$ of $x$ and it is positive recurrent if for every such
$U$, $N(x,U) \cap \N$ is infinite.
If $\pi : (X,T) \to (X_1,T_1)$
is a surjective morphism and $x$ is a recurrent point in $X$ then $\pi(x)$ is a recurrent point in $X_1$.

\begin{lem}\label{leminv} If $(X,T)$ is the inverse limit of a system $\{(X_i,T_i)\}$ indexed by the directed set $I$ and with
$\pi_i : X \to X_i$ the corresponding projections, then $x \in X$ is recurrent (or positive recurrent) if and only if
$\pi_i(x)$ is recurrent (resp. positive recurrent) for every $i \in I$. \end{lem}

 \begin{proof}
 If $px = x$ for some $p \in Ad(X)$ then $p\pi_i(x) = \pi_i(px) = \pi_i(x)$ for all $i$.

 Conversely, if $\pi_i(x)$ is recurrent, then $K_i = \{ p \in Ad(X) : p \pi_i(x) = \pi_i(x) \}$ is a closed, non-empty
 compact subset of $Ad(X)$. Since $I$ is directed, these sets form a filterbase and so by compactness the intersection
 $K$ is a nonempty subset of $Ad(X)$ which consists of the $p \in Ad(X)$ such that $px = x$.

 For positive recurrence apply the same argument using $Ad_+$ instead.
\end{proof}

The system $(X,T)$ is called \emph{pointwise recurrent} when every point of $X$ is
recurrent, i.e. $x \in \al(x) \cup \om(x)$ for all $x \in X$. It is called \emph{pointwise forward recurrent}
when $x \in \om(x)$ for all $x \in X$.

It is clear that  subsystems and factors of a system which is pointwise recurrent or pointwise forward recurrent satisfy the corresponding
property.  By Lemma \ref{leminv} these properties are preserved by inverse limits as well.

\begin{lem}\label{pointwiserec} If the product system $X \times X$ is pointwise recurrent then either
$(X,T)$ or $(X,T^{-1})$ (or both) is pointwise forward recurrent. \end{lem}

 \begin{proof}
 If neither $(X,T)$ nor $(X,T^{-1})$  is pointwise forward recurrent, then there exist $x, y \in X$ such that
 $x \notin \om(x)$ and $y \notin \al(y)$. Since $\om(x,y) \subset \om(x) \times \om(y)$ and similarly for $\al$ it follows that
 $(x,y)\notin \al(x,y) \cup \om(x,y)$.
 \end{proof}

\begin{defn}
Given a cascade $(X, T)$, let $A$ be an invariant subset of $X$.
\begin{enumerate}
\item[(a)] We say that  $A$ is an {\em isolated invariant set}
if there is a closed subset $U \subset X$ containing $A$ in its interior such that
$A$ is the maximum closed invariant subset of $U$, i.e.
if a closed invariant set $B$ is contained in $U$ then $B \subset A$; or, equivalently
if $A = \bigcap_{n \in \Z}  T^n(U)$. The set $U$ is called an \emph{isolating neighborhood} for $A$.

\item[(b)] We will call  $A$  an {\em attractor}
if there is a closed subset $U \subset X$ containing $A$ in its interior such that
 $A = \bigcap_{n \in \N}  T^n(U)$.
 \end{enumerate}
\end{defn}

 Thus, an attractor is isolated and an isolated invariant set is closed.

 The above is not the standard definition of an attractor but is equivalent to it, see \cite{Ak-93}
Theorem 3.3(b). If $A$ is an attractor for $(X,T)$ then there exists a \emph{dual repeller} $R$, i.e. an
attractor for $(X,T^{-1})$, disjoint from $A$ and such that for all $x \in X \setminus (A \cup R)$, $\om(x) \subset A$ and $\al(x) \subset R$.
In particular, no point of $X \setminus (A \cup R)$ is recurrent. See \cite{Ak-93} Proposition 3.9. From this we prove the following result from
\cite{Ak-96}.

\begin{thm}\label{iso} If $(X,T)$ is pointwise forward recurrent  and $A$ is an isolated invariant subset of $X$ then $A$ is
clopen in $X$. \end{thm}

 \begin{proof}
 Let $U$ be an isolating neighborhood for $A$. If $x \in \bigcap_{n \in \N} T^{-n}(U)$ then $T^{n}(x) \in U$ for all $n \in \N$.
 Since $U$ is closed, $\om(x) \subset U$. Since $\om(x)$ is invariant, it is a subset of $A$. Because the system is pointwise
 forward recurrent, $x \in \om(x)$ and so $x \in A$.  Thus, $A = \bigcap_{n \in \N} T^{-n}(U)$. Thus, $A$ is an attractor
 for $(X,T^{-1})$ with dual repeller $R$. Since no point of $X \setminus (A \cup R)$ is recurrent, this set must be empty.
 Thus, $A = X \setminus R$ is open.
 \end{proof}

 \begin{exas}
 (a)
 Let $L = IP\{10^t\}_{t=1}^\infty \subset \N$ be the IP-sequence generated by
the powers of ten, i.e.
$$
L =\{10^{a_1} + 10^{a_2} + \cdots + 10^{a_k} :
1 \le a_1 < a_2 < \cdots < a_k\}.
$$
Let $f= 1_L$ and let $X= \overline{\mathcal{O}_T(f)} \subset \{0,1\}^\Z$, where $T$ is
the shift on $\Om = \{0,1\}^\Z$.
 It is easy to check that $X$ is a Cantor set with a single fixed point ${\bf{0}}$ and such that
 every other orbit is dense, so that $(X,T)$ is pointwise recurrent.
 The fixed point ${\bf{0}}$ forms an isolated invariant set which is not clopen and
 it follows that neither $(X,T)$ nor $(X,T^{-1})$ is pointwise
 forward recurrent, so $(X \times X, T \times T)$ is not pointwise recurrent.
 In fact, the point $f$ is clearly forward but not backward recurrent and there are points in $X$ which are
 backward but not forward recurrent. (See also \cite{DY} and \cite[Theorem 4.16]{Ak-16}.)


 (b)
 As was shown in \cite[page 180]{Ak-93} the stopped torus example $(X,T)$ is a connected,
topologically mixing, pointwise recurrent
system, containing a unique fixed point as its mincenter
\footnote{The mincenter of a dynamical system is the closure
of the union of its minimal subsystems.}. 
Furthermore, the fixed point is a proper
isolated closed invariant subset.
So again neither $(X,T)$ nor $(X,T^{-1})$ is pointwise  forward recurrent.

%
%
 (c)
 Let $(X,T)$ be an infinite subshift. Such a system is \emph{expansive}
 which says exactly that the diagonal $\Delta_X$ is an
 isolated, invariant subset of $X \times X$.
 Since $X$ is infinite, the diagonal is not clopen and so neither $(X \times X, T \times T)$
 nor $(X \times X, (T \times T)^{-1})$ is pointwise forward recurrent. Hence, the product system on $X^4$ is not pointwise recurrent.
 On the other hand, if, for example, $(X,T)$ is minimal then it and its inverse are pointwise forward recurrent.

 (d)
 A minimal system $(X,T)$ is called {\em doubly minimal} if every point in
 $$
 X \times X \setminus \bigcup \{({\Id} \times T)^n \Delta_X : n \in \Z\}
 $$
 has a dense orbit in $X\times X$.
 Such systems exist in abundance, and many of them are also subshifts, hence expansive (see \cite{W-98}).
 It is shown in \cite{GW} that  if $X$ is
 a doubly minimal subshift, then $X \times X$ is pointwise recurrent.
 Thus such a system $(X,T)$ has the property that $X \times X$ is pointwise recurrent but
 $X^4$ is not.
 \end{exas}

The system $(X,T)$ is called \emph{weakly rigid}
(see \cite{GMa}) when $\Id_X$ is a recurrent point of $E(X)$. That is, $Id_X \in Ad(X)$ and so either
$Id_X \in Ad_+(X)$ or $Id_X \in Ad_-(X)$.

\begin{thm}\label{WR}
For a cascade $(X,T)$ the following are equivalent.
\begin{itemize}
\item[(i)]  The system $(X,T)$ is weakly rigid.
\item[(ii)] $E(X) = Ad(X)$.
\item[(iii)]  The product system $(X^I,T^{(I)})$ is weakly rigid for every nonempty index set $I$.
\item[(iv)] The product system $(X^I,T^{(I)})$ is pointwise recurrent for every nonempty index set $I$.
\item[(v)] $(E(X),T_*)$ is pointwise recurrent.
\item[(vi)] For any $k$-tuple $(x_1, x_2, \dots, x_k) \in X^k$ the subsystem
$\overline{\mathcal{O}(x_1, x_2, \dots, x_k)}$ is pointwise recurrent.
\item[(vii)] The collection $\{ N(x, U) : x \in X, U \ {\text{an open set with}} \ x \in U \}$ forms a filter base.
\end{itemize}
\end{thm}

 \begin{proof}
  (i) $\Leftrightarrow$ (ii):  As it is the orbit closure of $Id_X$, $E(X)$  is contained in the closed, invariant set
  $Ad(X)$ if and only if $Id_X \in Ad(X)$.

  (i) $\Leftrightarrow$ (iii): $E(X,T)  \cong  E(X^I,T^I)$.

 (i) $\Rightarrow$ (vi): The evaluation map at $(x_1,\dots,x_k)$ takes $Id_X$ to
 $(x_1,\dots,x_k)$. Similarly, (iii) $\Rightarrow$ (iv).

  (iv) $\Rightarrow$ (v): $E(X)$ is a subsystem of $X^X$.

  (v) $\Rightarrow$ (i): Obvious.

   (vi) $\Rightarrow$ (v): $E(X)$ is the inverse limit of the $\overline{\mathcal{O}(x_1, x_2, \dots, x_k)}$'s.

    (vi) $\Leftrightarrow$ (vii): If $U_i$ is an open set containing $x_i$ for $i = 1, \dots, k$, then
    $N((x_1,\dots,x_k),U_1 \times \dots U_k) = \bigcap_{i=1}^k N(x_i,U_i)$.
     \end{proof}

It follows that weak rigidity is preserved by factors, subsystems and inverse limits.

When $Id_X \in Ad_+(X)$ then $E(X) = Ad_+(X)$ and every $(X^I,T^{(I)})$ is pointwise forward recurrent. Conversely, if
every $\overline{\mathcal{O}(x_1, x_2, \dots, x_k)}$ is pointwise forward recurrent then $Id_X \in Ad_+(X)$. In particular, if
$X$ is weakly rigid then either $(X,T)$ or $(X,T^{-1})$ is pointwise forward recurrent.

Recall that $x \in X$ is almost periodic, or  a.p., when it has a minimal orbit closure. We call the system $(X,T)$ \emph{pointwise almost periodic} or pointwise a.p. when every point $x \in X$ is a.p. A point transitive, pointwise a.p. system is minimal.

We will call the system \emph{distal} when $(X \times X, T \times T)$ is
pointwise a.p. Again this is not the usual definition but the equivalence is described in
\cite[Theorem 1]{Ellis-58}
(see also \cite{Aus} Theorem 5.6).
Distality is
preserved by factors, subsystems and inverse limits.
Using
\cite[Theorem 1]{Ellis-58} or \cite{Aus} Theorem 5.6
and arguments similar to those of Theorem \ref{WR}
one can prove the following.

\begin{thm}\label{Dis}
For a cascade $(X,T)$ the following are equivalent.
\begin{itemize}
\item[(i)]  The system $(X,T)$ is distal.
\item[(ii)] $(E(X),T_*)$ is a minimal system.
\item[(iii)] $E(X)$ is a group.
\item[(iv)]  The product system $(X^I,T^{(I)})$ is pointwise a.p. for every nonempty index set $I$.
\item[(v)] For any $k$-tuple $(x_1, x_2, \dots, x_k) \in X^k$ the subsystem
$\overline{\mathcal{O}(x_1, x_2, \dots, x_k)}$ is pointwise recurrent.
\item[(vi)] There exists a filter $\mathcal{F}$ of syndetic sets such that for every $x \in X$ and
open subset $U$ containing $x$, the return time set $N(x,U)$ is an element of $\mathcal{F}$.
\end{itemize}
\end{thm}

\begin{cor}\label{cor0}
 If system $(X,T)$ is distal, then it is weakly rigid. Moreover, $E(X) = Ad_+(X) = Ad_-(X)$. \end{cor}

  \begin{proof}
  Since $E(X)$ is minimal, it is equal to each of the nonempty, closed invariant subsets $Ad_+(X)$ and $Ad_-(X)$.
  In particular, $Id_X \in E(X)$ lies in $Ad_+(X)$ and $Ad_-(X)$.
   \end{proof}

The system $X$ is {\em weak mixing} (hereafter WM) when the product system $X \times X$ is transitive.
The transitivity property is preserved by factors and
 inverse limits and so the same is true for WM.

Recall that subset of $\Z$ is called {\em thick} when it contains runs of arbitrary length.
The following is a classic result of Furstenberg (see \cite[Proposition II.3]{F}).

\begin{thm}\label{FT}
For a cascade $(X,T)$ the following are equivalent.
\begin{itemize}
\item[(i)]  The system $(X,T)$ is WM.
\item[(ii)]  The product system $(X^I,T^{(I)})$ is WM for every nonempty index set $I$.
\item[(iii)] For every pair $U,V$ of nonempty open subsets of $X$, the visiting time set $N(U,V)$ is thick.
\item[(iv)] The system $(X,T)$ is transitive and for every nonempty open subset $U$ of $X$,
the return time set $N(U,U)$ is thick.
\item[(v)] The collection $\{ N(U, V) : U, V \ {\text{non-empty open sets}}\}$ forms a filter base.
\end{itemize}
\end{thm}

\begin{cor} If a cascade $(X,T)$ is WM, then it it totally transitive. \end{cor}

  \begin{proof}
  Every $N(U,V)$ is thick and so meets $k \Z$ for any positive integer $k$.
  \end{proof}

\begin{lem}\label{lem3}
Let $(X,T)$ be a nontrivial, WM  system. Then every transitive point in $X$
is recurrent.
\end{lem}

\begin{proof}
Let $x \in X$ be a transitive point.
If $x$ is not recurrent, then it has a neighborhood $U$ such that $N(x,U)$ is finite.
As $x$ is not recurrent, it is not periodic and so we can remove a finite set to obtain a neighborhood $U'$ such that
$N(x,U') = \{ 0 \}$. As $N(x,U' \setminus \{ x \}) = \emptyset$ and $x$ is a transitive point it follows that
$U' \setminus \{ x \} = \emptyset$. That is,
the singleton set $\{x \}$ is clopen. Now, in this case, the sets $U =\{(x, x)\}$ and $V = \{(x, Tx)\}$
are clopen subsets of $X \times X$ but there is no $k \in \Z$ such that $(T \times T)^kU \cap V \not= \emptyset$.
Thus, $(X,T)$ is not weak mixing.
\end{proof}

\section{Some obstructions to WM of $E(X,T)$}\label{sec-2}

Recall that an ambit  $(X,x_0,T)$ is an enveloping semigroup if and only if it is {\em point
universal}; i.e. it satisfies the following condition:
For every $x \in X$ there is a (unique) homomorphism of ambits $(X,x_0,T)
\to (X,x,T)$ (see e.g. \cite[Proposition 2.6]{GMe}).

We will say that $X$ \emph{has a WM enveloping semigroup} when the system $(E(X),T_*)$ is WM.

Call a subset $F$ of $\Z$ \emph{diff-thick} when the difference set $\{ i - j : i, j \in F \}$ is thick.

\begin{thm} For a cascade $(X,T)$ the following are equivalent.
\begin{itemize}
\item[(i)]  The system $(X,T)$ has a WM enveloping semigroup.
\item[(ii)]  For any $k$-tuple $(x_1, x_2, \dots, x_k) \in X^k$ the ambit
$\overline{\mathcal{O}(x_1, x_2, \dots, x_k)}$ is WM.
\item[(iii)]
There exists a filter $\mathcal{F}$ of diff-thick sets such that for every $x \in X$ and
open subset $U$ containing $x$, the return time set $N(x,U)$ is an element of $\mathcal{F}$.
\end{itemize}
\end{thm}

\begin{proof} (i) $\Leftrightarrow$ (ii): Each ambit  $\overline{\mathcal{O}(x_1, x_2, \dots, x_k)}$ is a factor of $E(X)$ and $E(X)$ is an inverse
limit of these factors.  WM is preserved by factors and inverse limits.

(iii) $\Rightarrow$ (ii):
If $U_1, U_2$ are open sets which meet $\overline{\mathcal{O}(z)}$ with $z = (x_1, x_2, \dots, x_k)$ then there exist
$i_1, i_2 \in \Z$ such that
$(T^{(k)})^{i_1}(z) \in U_1, (T^{(k)})^{i_2}(z) \in U_2$. Let $U = (T^{(k)})^{-i_1}(U_1) \cap (T^{(k)})^{-i_2}(U_2)$.
Since $\mathcal{F}$ is a filter, $N(z,U) \in \mathcal{F}$.
If $ (T^{(k)})^{k_1}(z), (T^{(k)})^{k_2}(z) \in U$ then $k_2 - k_1 + (i_2 - i_1) \in N(U_1,U_2)$.
Since $N(z,U)$ is diff-thick and the translate of a thick set is thick, $N(U_1,U_2)$ is thick.
Hence, $\overline{\mathcal{O}(z)}$ is WM.

(ii) $\Rightarrow$ (iii):
If $U_{\ell}$ is a neighborhood of $x_{\ell}$ for $\ell = 1, \dots, k$
and $U = U_1 \times \dots \times U_k$ then
$U$ is a neighborhood of $z = (x_1, x_2, \dots, x_k)$ and so $N(U,U)$ is thick.
 Furthermore, as above, $N(U,U) = N(z,U) - N(z,U)$ and
so $N(z,U)$ is diff-thick.
Since $N(z,U) = \bigcap_{\ell = 1}^{k} N(x_{\ell},U_{\ell})$, it follows that $\{ N(x,V) : x \in X, V$ open in $X$ with
$ x \in V \}$ generates a filter of diff-thick sets.

\end{proof}

So the following is a consequence of Lemma \ref{lem3}.

\begin{cor}\label{cor}
 If system $(X,T)$ has a WM enveloping semigroup then it is weakly rigid. \end{cor}

 \br

\begin{thm}\label{main}
Let $(X,T)$ be a transitive cascade.
\begin{enumerate}
\item
If $(X,T)$ is weakly rigid, or, more generally, if $(X \times X, T \times T)$ is pointwise
recurrent, then $X$ does not admit a proper isolated, closed, invariant subset.
\item
If $(X,T)$ is weakly rigid and WM, then $X$ is connected.
\end{enumerate}
\end{thm}

\begin{proof}

If $(X, T)$ is weakly rigid, then
$(X \times X, T \times T)$ is pointwise recurrent.
By Lemma  \ref{pointwiserec}
 either $(X,T)$ or $(X,T^{-1})$ is pointwise forward recurrent.
By  Theorem \ref{iso}  an isolated, closed, invariant proper subset
would be clopen, contradicting  the transitivity of $X$.

Now suppose that $X$ is not connected and so there exist $U, V$
proper, disjoint clopen sets with union $X$.
Then $W = (U \times U) \cup (V \times V)$ is a proper, clopen subset of $X \times X$
containing the diagonal $\Delta_X = \{(x,x) : x \in X\}$, which is a non-empty invariant set.
Hence, $A = \bigcap_{n \in \Z}  (T \times T)^n(W)$ is a non-empty, closed, invariant set since
$W$ is closed. Since $W$ is open, it is an isolating neighborhood for $A$.
If $X$ were WM then $X \times X$ would be transitive and,
being weakly rigid, $X \times X \times X \times X$ is pointwise recurrent
and so $X \times X$ can contain no such isolated invariant set.
\end{proof}


\begin{cor}
Let $(X,T)$ be a transitive  cascade.
\begin{enumerate}
\item
If $X$ is not connected, then its enveloping semigroup $E(X,T)$, as a dynamical system, is
not WM.
\item
If $X$ admits a proper, isolated, closed, invariant subset, then its enveloping semigroup $E(X,T)$,
as a dynamical system, is not WM.
\end{enumerate}
\end{cor}

\begin{proof}
By Corollary \ref{cor},  $(X,T)$ is  weakly rigid and WM if $E(X,T)$ is WM. Now apply Theorem \ref{main}.
\end{proof}

\begin{rmk}\label{second}
 We will next give an alternative proof of
the fact that a weakly rigid WM system is necessarily connected; in fact we prove a slightly
stronger result.
\end{rmk}

\begin{thm}\label{Tsecond}
A totally transitive,
weakly rigid cascade $(X,T)$ is connected.
\end{thm}

\begin{proof}
Assume to the contrary that $X$ is not connected.
We first note that $\hat{X}$, the canonically defined largest totally disconnected factor of $X$, is
nontrivial, and that $E(X) \to E(\hat{X})$.
So we now assume that $X$ is nontrivial and totally disconnected.
Such a system always admits
a nontrivial symbolic factor $X \to Y$
(i.e. $Y \subset \{0,1\}^\Z$ is a subshift), which by total transitivity is infinite.
It then follows that there are four distinct points $y_i \  (i =1,2,3,4)$
such that $y_1$ and $y_2$ are right asymptotic while $y_3$ and $y_4$ are left asymptotic
(see \cite[Theorem 10.36]{GH}, the so called Schwartzman lemma).
Now clearly any limit point of the orbit $\{T^n(y_1, y_2, y_3, y_4): n \in \Z\}$ has at most
three distinct coordinates. In particular the point $(y_1, y_2, y_3, y_4)$ is not recurrent.
This implies that $Y$ is not weakly rigid,
contradicting the fact that $Y$ is a factor of $X$.
(See also Proposition 6.7 in \cite{GMa}.)
\end{proof}

\section{The horocycle flow}

\begin{thm}
The enveloping semigroup of a classical horocycle flow $(X, \{U_t\}_{t \in \R})$,
where $G = PSL(2, \R)$, $\Gamma < G$ is a uniform lattice, $X = G / \Gamma$ and
$U_t(g\Ga) = u_t g \Ga = \begin{psmallmatrix}1 & t\\0 & 1\end{psmallmatrix} g\Gamma, \ t\in \R, \ g \in G$,
is WM.  The same holds for the discrete flow $(X, U_1)$.
\end{thm}

\begin{proof}
Recall that $E(X, \{U_t\}_{t \in \R})$ is isomorphic, as a flow, to the infinite pointed product
of the family of pointed systems $\{(X,x) : x \in X\}$, i.e.
$E(X, \{U_t\}_{t \in \R}) = \bigvee_{x \in X} (X,x)
= \overline {\{U_t {\bf{x}} : t \in \R\}} \subset X^X$, where ${\bf{x}} \in X^X$ is the identity map
${\bf{x}}(x) = x, \  \forall x \in X$.
It follows that $E(X, \{U_t\}_{t \in \R})$ is the inverse limit of the family of finite pointed systems
$$
\{X(\{x_1, x_2, \dots, x_n\}) = (X, x_1) \vee (X, x_2) \vee \cdots \vee(X, x_n)\},
$$
where we range over the directed collection
of, unordered, $k$-tuples $\{x_1, x_2, \dots, x_n\} \subset X$.
It therefore suffices to show that each $X(\{x_1, x_2, \dots, x_n\})$ is WM.
The fact that this is indeed the case is a direct corollary of Ratner's theory, as follows:

By Ratner's orbit closure theorem we have
$$
X(\{x_1, x_2, \dots, x_n\}) = \overline{\{U_t X(x_1, x_2, \dots, x_n) : t \in \R\}} =
H(x_1, x_2, \dots, x_n),
$$
where $H < G \times G \times \cdots \times G$ ($n$ times) is a closed connected subgroup
of $G^n$ containing the subgroup
$\{u_t \times u_t \times \cdots \times u_t  : t \in \R\}$
and there is a discrete uniform lattice $\La < H$ so that
$X(\{x_1, x_2, \dots, x_n\}) = H(x_1, x_2, \dots, x_n) \cong H/ \La$.
By unique ergodicity of $(G/ \Ga,  \{U_t\}_{t \in \R})$, the Haar measure $\la$ on $H/ \La$ is an $n$-fold
self-joining of $\mu$, the unique invariant probability measure on $G/\Ga$.
Now apply Ratner's joining theorem (see  \cite[Theorem 7, page 283]{Ra-83})
to conclude that $\la$
%
%
%
%
has the following form:
There is, for some $q, \ 1 \le q \le n$, a partition
\begin{gather*}
I_j \subset I = \{1,2,\dots, n\}, \ I_j = \{i_1^{(j)}, i_2^{(j)},   \dots, i_{n_j}^{(j)}\},\ j = 1,\dots q, \\
\bigcup_{j=1}^q I_j = I,
\end{gather*}
and there are elements $a_{i_k}^{(j)} \in G, \ k=1, \dots, n_j, \ j=1,\dots,q$, with
$$
\Gamma_0^{(j)} = \bigcap_{k =1}^{n_j} a_{i_k}^{(j)} \Gamma (a_{i_k}^{(j)})^{-1}
$$
uniform lattices, such that,
each $(X^{I_j},\la_j, \{U_t\}_{t \in \R})$
is isomorphic to the horocycle flow on
$(G/ \Gamma_0^{(j)}, \mu_0^{(j)}, \{U_t\}_{t \in \R})$,
and $\la = \prod_{j=1}^q \la_j$.
%
%
In particular then the measure preserving dynamical system
$(H/\La, \la, \{U_t\}_{t \in \R})$ is measure theoretically weakly mixing, hence also
topologically WM.

Finally, the same arguments will work for the discrete flow $(X, U_1)$.
\end{proof}

\begin{exa}
The case where $\Gamma$ is a nonarithmetic maximal uniform lattice is special. We see that
$E(X,   \{U_t\}_{t \in \R})  = E(G/\Ga, \{U_t\}_{t \in \R}) \cong X^{[X]}$, as dynamical systems,
where $[X]$ denotes the collection of $\{U_t\}_{t \in \R}$ orbits in $G/\Ga$.
Note that both the $\R$-flow $E(X,   \{U_t\}_{t \in \R})$ and the discrete system
$E(X, U_1)$, have the property that any point ${\bf{x}} \in X^{[X]}$ whose
coordinates have the property that no two of them belong to the same
$\{U_t\}_{t \in \R}$ orbit, has a dense orbit in $X^{[X]}$.
Of course
$\card ([X]) = \mathfrak{c}$, the cardinality of the continuum.
Projecting on any set of four coordinates we obtain a real flow with $4$-fold minimal self joinings.
The corresponding $U_1$ cascade has the same property when we allow for off-diagonals resulting
from the $\{U_t\}_{t \in \R}$ flow.
This does not contradict the result of J. King \cite[page 756]{K-90}, which says that  there is no infinite
minimal cascade $(X, T)$ with $4$-fold minimal self-joinings.
The reason is that here, as it turns out, the `future $\ep$-bounded pair' produced in
Theorem 21 of \cite{K-90}, must lie on the same $\{U_t\}_{t \in \R}$ orbit.
\end{exa}

\begin{rmk}
It is easy to see that if $(X,T)$ is a dynamical system whose enveloping semigroup
is WM then so is the associated system $(\hat{X}, T)$ obtained from $X$ by collapsing
its mincenter to a point
(see e.g. the remark at the beginning of Section \ref{sec-2}).
This leads us to the following:
\end{rmk}

\begin{prob}
Is there a nontrivial metric connected cascade $(X,T)$ with trivial mincenter
whose enveloping semigroup is WM ?
Note that such a system is necessarily both proximal and weakly rigid.
\end{prob}

\br

\end{document}